\documentclass[a4paper, 12pt]{amsart}
\usepackage[active]{srcltx}
\usepackage[all]{xy}
\usepackage{subfig}
\usepackage{hyperref}
\hypersetup{
colorlinks,
allcolors=blue
}
\usepackage{mathrsfs}

\usepackage{amsmath,amssymb,amsthm,a4wide}
\usepackage{enumerate}
\usepackage{color}

\theoremstyle{plain}
\newtheorem{lemma}{Lemma}[section]
\newtheorem{theorem}[lemma]{Theorem}
\newtheorem{proposition}[lemma]{Proposition}

\newtheorem{conjecture}[lemma]{Conjecture}

\theoremstyle{definition}

\newtheorem{definition}[lemma]{Definition}
\newtheorem{remark}[lemma]{Remark}

\theoremstyle{remark}

\newtheorem*{ack}{Acknowledgements}

\numberwithin{equation}{section}

\newcommand{\s}{\mathbb{S}}

\newcommand{\R}{\mathbb{R}}

\newcommand{\RR}{\mathbb{R}}

\newcommand{\cP}{\mathcal{P}}

\newcommand{\proj}{{\mathsf{pr}}}

\newcommand{\RP}{\mathbb{R}P}

\newcommand{\sphere}{\mathrm{\mathbb{S}}}

\DeclareMathOperator{\RCD}{\mathsf{RCD}}
\DeclareMathOperator{\CD}{\mathsf{CD}}

\DeclareMathOperator{\Haus}{\mathcal{H}}

\DeclareMathOperator{\supp}{supp}
\DeclareMathOperator{\Susp}{Susp}
\DeclareMathOperator{\vol}{vol}

\newcommand{\be}{\begin{equation}}
\newcommand{\ee}{\end{equation}}

\numberwithin{equation}{section}

\begin{document}


\title[Alexandrov spaces with positive or nonnegative Ricci curvature]{Three-Dimensional Alexandrov spaces with positive or nonnegative Ricci curvature}


\author[Deng]{Qintao Deng$^*$}
\address[Deng]{Central China Normal University, School of Mathematics and Statistics \& Hubei Key Laboratory of Mathematical Sciences , 152 Luoyu Road, Wuhan 430079, P.R. China}
\email{qintaodeng@yeah.net}
\thanks{$^{*}$ Supported in part by NSFC (no.11171126) and the Special Fund for Basic Scientific Research of Central Colleges (no. CCNU10A02015 and no. CCNU14A05038).}


\author[Galaz-Garc\'ia]{Fernando Galaz-Garc\'ia$^{**}$}
\address[Galaz-Garc\'ia]{Karlsruher Institut f\"ur Technologie (KIT), Institut f\"ur Algebra und Geometrie, Englerstr. 2, 76131 Karlsruhe, Germany}
\email{galazgarcia@kit.edu}
\thanks{$^{**}$ Supported in part by research grant MTM2014-57769-C3-3-P from the Ministerio de Econom\'ia y Competitividad of Spain.} 


\author[Guijarro]{Luis Guijarro$^{***}$}
\address[Guijarro]{Department of Mathematics, Universidad Aut\'onoma de Madrid, and ICMAT CSIC-UAM-UCM-UC3M, Spain}
\email{luis.guijarro@uam.es}
\thanks{$^{***}$ Supported by research grants MTM2011-22612, MTM2014-57769-C3-3-P from the Ministerio de Econom\'ia y Competitividad and MINECO: ICMAT Severo Ochoa project SEV-2011-0087.}


\author[Munn]{Michael Munn}
\address[Munn]{New York University, Courant Institute, 251 Mercer St, New York, NY 10012}
\email{munn@nyu.edu}


\begin{abstract}We study closed three-dimensional Alexandrov spaces with a lower Ricci curvature bound in the $\CD^*(K,N)$ sense, focusing our attention on those with positive or nonnegative Ricci curvature. First, we show that a closed three-dimensional $\CD^*(2,3)$-Alexandrov space must be homeomorphic to a spherical space form or to the suspension of $\mathbb{R}P^2$. We then classify closed three-dimensional $\CD^*(0,3)$-Alexandrov spaces. 
\end{abstract}

\date{\today}



\maketitle

\section{Introduction}
Alexandrov spaces can be viewed as metric geometry generalizations of Riemannian manifolds with a lower sectional curvature bound. They arise as the Gromov-Hausdorff limit of this class of manifolds or when taking quotients by isometric Lie group actions with closed orbits. As such, they provide a natural setting in which to study many questions in global Riemannian geometry and by now there is a large literature concerning their properties (see \cite{BBI} and the references therein). 

While definitions for a lower sectional curvature bound make sense on a metric space $(X,d)$, it turns out that to capture the notion of a lower Ricci curvature bound one requires the additional structure of a measure $m$ on $X$. There have been various approaches to defining Ricci lower bounds for metric measure spaces $(X,d,m)$ (see, for example, \cite{BakryEmery,CheegerColding,LV,SturmI,SturmII}). Here, we focus on the so-called \emph{curvature-dimension condition} defined via optimal transportation of probability measures on the space $X$. Sturm \cite{SturmI, SturmII} and Lott-Villani \cite{LV} independently introduced the curvature dimension condition $\CD(K,N)$ which provides a notion of Ricci curvature bounded below by $K \in \mathbb{R}$ and dimension bounded above by $N \in (1, \infty]$. 

Over time, these definitions have been further adapted through the work of Bacher and Sturm \cite{BacherSturm}, and Ambrosio, Gigli and Savar\'e \cite{AGS} and there has been considerable effort to better understand the structure this condition imposes on a space (see, for instance, \cite{AGS, MondinoNaber} among many others). In particular, the $\CD^*(K,N)$ condition, introduced by Bacher-Sturm in \cite{BacherSturm}, is a sort of local variant of the $\CD(K,N)$ and slightly  weaker than the usual, global curvature-dimension condition. This so-called \emph{reduced curvature dimension condition} has the local-to-global property on non-branching metric measure spaces which is a necessary component of our argument.

In \cite{GGG}, the second and third authors classified closed three-dimensional Alexandrov spaces with positive or non-negative curvature. In the present paper, we classify Alexandrov  spaces of dimension $3$ under the weaker hypothesis of positive or non-negative Ricci curvature. This extends the classification of closed, smooth three-manifolds with positive or nonnegative Ricci curvature to the class of Alexandrov spaces.
Throughout this paper, we use the following terminology:


\begin{definition}
A metric space $(X,d)$ is said to be a \emph{$\CD^*(K,N)$-Alexandrov space} provided $(X,d)$ is an $N$-dimensional Alexandrov space of curvature bounded below by some $k \in \mathbb{R}$ and the Hausdorff measure $\Haus^N$ satisfies the $\CD^*(K,N)$ condition (see Definition \ref{DEF:CDSTAR}). Define a $\CD(K,N)$-Alexandrov space similarly. 
\end{definition}

We first classify closed three-dimensional Alexandrov spaces with positive Ricci curvature.


\begin{theorem} 
\label{THM:RICCI_POS}
Let $(X^3,d)$ be a closed $\CD^*(2,3)$-Alexandrov space.
\begin{enumerate}
	\item If $X^3$ is a topological manifold, then it is homeomorphic to a spherical space form.\smallskip
	\item If $X^3$ is not a topological manifold, then it is homeomorphic to $\Susp(\RP^2)$, the suspension of $\RP^2$. 
\end{enumerate}
\end{theorem}

Our Theorem \ref{THM:RICCI_POS}  corresponds to the classification of closed Riemannian 3-manifolds with positive Ricci curvature given by Hamilton in \cite{Hamilton}. Note that while, for the class of smooth Riemannian manifolds, this classification is up to diffeomorphism, the lack of differentiable structure for metric measure spaces (or even Alexandrov spaces) means we obtain classification only up to homeomorphism. The proof of Theorem \ref{THM:RICCI_POS} is contained in Section \ref{SEC:RICCI_POS}.

As mentioned above, our proof generalizes the result of Galaz-Garc\'ia and Guijarro (see also \cite{HarveySearle}) by adapting their argument to the metric measure space setting. Petrunin \cite{Petrunin} proved that the lower Alexandrov curvature bound is compatible with the curvature-dimension condition for lower Ricci curvature bound given by Lott-Sturm-Villani. As such, our assumption of a lower Ricci bound is an inherently weaker condition and we prove various additional properties to fully employ their argument in our setting.

The following theorem deals with the nonnegatively curved case. The homeomorphism statements are also due to the lack of a differentiable structure;  nonetheless, when the Alexandrov space is flat, we get classification up to isometry.


\begin{theorem}
\label{THM:RICCI_NONNEG}

Let $(X^3,d)$ be a closed $\CD^*(0,3)$-Alexandrov space. 
\begin{enumerate}
\item If $X^3$ is a topological manifold, then one of the following holds:
	\begin{itemize}
		\item $X^3$ is homeomorphic to a spherical space form,
		\item $X^3$ is homeomorphic to $\s^2 \times \s^1$, $\RP^2 \times \s^1$, $\RP^3 \# \RP^3$ or $\s^2 \tilde{\times} \s^1$,\\
		or
		\item $X^3$ is isometric to a closed, flat three-dimensional  space form\\
	\end{itemize}
\item If $X^3$ is not a topological manifold, then either: 
	\begin{itemize}
	\item $X^3$ is homeomorphic to $\Susp(\RP^2)$, $\Susp(\RP^2) \# \Susp(\RP^2)$ or 
	\item $X^3$ is isometric to a quotient of a closed, orientable, flat-three-dimensional manifold by an orientation reversing isometric involution with only isolated fixed points. 
	\end{itemize}
\end{enumerate}
\end{theorem}
The proof of Theorem \ref{THM:RICCI_NONNEG} is contained in Section \ref{SEC:RICCI_NONNEG}.

 Observe that the spaces in Theorems~\ref{THM:RICCI_POS} and \ref{THM:RICCI_NONNEG} are exactly those appearing in the classification of closed, three-dimensional Alexandrov spaces with positive or nonnegative curvature (in the triangle comparison sense) (see \cite[Theorem~1.3 and Corollary~2.2]{GGG}).

Our approach of studying $\CD^*(K,N)$-Alexandrov spaces rather than general $\CD^*(K,N)$ metric measure spaces has various advantages, for example, in Alexandrov spaces geodesics do not branch. 
While Rajala-Sturm \cite{RajalaSturm}  have shown that $\RCD(K,N)$-spaces are \emph{essentially non-branching} there is still much that is not known about the local structure and topology of $\RCD$-spaces (though we refer the reader to recent  work of Mondino-Naber \cite{MondinoNaber} for some results in this area). Indeed, it seems questions of a topological nature are not well-posed on such spaces. On the other hand, Alexandrov spaces are sufficiently well behaved so that assuming a $\CD(K,N)$ condition on their Hausdorff measure yields significant topological conclusions.


\begin{ack} Q.~Deng, F.~Galaz-Garc\'ia and M.~Munn would like to thank the hospitality of the Hausdorff Research Institute for Mathematics at the Universit\"at Bonn during the Junior Hausdorff Trimester Program ``Optimal Transportation''. L.~Guijarro would like to thank the Differential Geometry Group of the Institut f\"ur Algebra und Geometrie at the Karlsruher Institut f\"ur Technology (KIT) for its hospitality. 
\end{ack}


\section{Background and Preliminaries}
\label{SEC:BACKGROUND}


\subsection{A brief overview of Alexandrov spaces} 
\label{SUBSEC:ALEX}In this subsection, we give a very brief overview of the relevant facts about  Alexandrov spaces that we will need for this paper (see also \cite{BBI} for alternate definitions and a more complete introduction to Alexandrov spaces).

A finite-dimensional {\it Alexandrov space} is a complete, locally compact, connected, length space $(X,d)$ which satisfies locally a lower curvature bound $k \in \mathbb{R}$ in an angle-comparison sense. Roughly speaking, geodesic triangles in $X$ are ``fatter'' than equivalent ones in the space form of constant curvature $k$. 
Recall that a  \emph{geodesic space} is a metric space such that any two points $p,q \in X$ can be joined by a rectifiable curve whose length is equal to $d(p,q)$. We call such a distance realizing curve a minimal geodesic and use $[pq]_X$ to denote a (not necessarily unique) minimal geodesic in $X$ joining $p$ and $q$. 

To describe the lower curvature bound, fix a number $k\in \mathbb{R}$. For any triple of points $p,q,r \in X$ (usually thought of as a triangle $\Delta pqr$), let $\tilde{\Delta}pqr$ denote the triangle $\Delta \bar{p}\bar{q}\bar{r}$ of points in the  $k$-plane such that $|\bar{p}\bar{q}|_{k} = d(p,q), |\bar{q}\bar{r}|_{k} = d(q,r)$ and $|\bar{p}\bar{r}|_{k} = d(p,r)$. 


\begin{definition}
\label{DEF:ALEX}
A finite-dimensional, complete, locally compact, length space $(X,d)$ has Alexandrov curvature $\geq k$ in an open set $U \subset X$ if, for each quadruple of points $(p; a,b,c)$ in $U$,  
\[\widetilde{\angle}_{k}apb + \widetilde{\angle}_{k}bpc +\widetilde{\angle}_{k}cpa \leq 2\pi,\]
where $\widetilde{\angle}_{k}apb$ is the comparison angle at $\bar{p}$ of a triangle $\tilde{\Delta}apc$ in the $k$-plane (define $\widetilde{\angle}_{k}bpc$, $\widetilde{\angle}_{k}cpa $ similarly).
\end{definition}
When not precisely specified, we will refer simply to an Alexandrov space as a metric space with Alexandrov curvature bounded below by $k \in \mathbb{R}$, where we generally consider $k$ to be some large negative number.

The Alexandrov curvature condition has strong consequences on the structure of the space. In particular, for a complete length space, the local condition above implies the same is true globally for any quadruple of points in $X$ (see \cite{BBI, BGP}).

For $p \in X$,  the {\em space of directions of $X$ at $p$}, denoted by $\Sigma_pX$, is the metric completion of the geodesic directions at $p\in X$. It is an Alexandrov space of dimension $\dim(X) -1$  with curvature greater than or equal to one. The \emph{tangent cone} of $X$ at $p$, denoted by $T_pX$, is by definition the Euclidean cone over $\Sigma_pX$. It agrees with the closure under the angle distance of the collection of reparametrized geodesics emanating from $p$ (modulo the obvious identification of such curves when they agree on some interval of the form $[0,\delta)$ for some $\delta>0$. We denote a geodesic joining two points $p,q\in X$ by $[pq]$.


\begin{definition}[{\em Singularities of Alexandrov Spaces}]
A point $p\in X$ is a {\it metric singular point} if $\Sigma_pX$ is not isometric to the unit sphere $\mathbb{S}^{n-1}$. For some $\delta>0$, a point is called  {\it $\delta$-singular} if $\Haus^{n-1}(\Sigma_pX) \leq \vol(\mathbb{S}^{n-1}) - \delta$. Let $S_X$ and $S_{\delta}$ denote, respectively, the set of metric singular points of $X$ and the set of $\delta$-singular points of $X$. Note that $S_X = \bigcup_{\delta >0} S_{\delta}$.
We say that a point $p$ is {\it topologically singular} if  $\Sigma_pX$ is not homeomorphic to a sphere. 
\end{definition} 

The set of topologically singular points of a finite-dimensional Alexandrov space $X$ (without boundary) has codimension at least three. Thus, dimension $3$ is the first dimension in which topological singularities  arise in Alexandrov spaces. 


\subsection{Double branched cover for 3-dimensional Alexandrov spaces}
Perelman's conical neighborhood theorem asserts that a small metric neighborhood of a point $p\in X$ is homeomorphic to the cone over $\Sigma_pX$ (see \cite{Perelman-Morse}). Therefore, a closed three-dimensional Alexandrov space that is not a topological manifold must have some point $p \in X$ whose space of directions $\Sigma_pX$ is not homeomorphic to $\s^2$. Since $\Sigma_pX$ is positively curved, it must be homeomorphic to  the real projective plane $\RP^2$.  It follows that $X$ can be exhibited as a compact  $3$-manifold whose boundary consists of finitely many $\RP^2$-components where one glues in cones over $\RP^2$.  Moreover, $X$ is the base of a two-fold branched cover $\proj: Y \to X$ whose total space $Y$ is a closed, orientable manifold and whose branching set consists of the  topologically singular points in $X$.

 
 \begin{lemma}[\protect{see \cite[Lemma 1.7]{GGG}}] 
 \label{LEM:BR_COV}
 Let $X$ be a closed three-dimensional Alexandrov space. If $X$ is not a topological manifold,  then there is a closed, orientable $3$-manifold $Y$ and an orientation reversing involution $\iota:Y\to Y$ with isolated fixed points such that $X$ is homeomorphic to the quotient $Y/\iota$. 
 \end{lemma}

Grove and Wilking proved that two-fold branched covers of Alexandrov $3$-spheres of nonnegative curvature with branching set an extremal knot also admit an Alexandrov metric with the same lower curvature bound as the base (cf.~\cite[Lemma~5.2.]{GW}). Their arguments also imply that the same holds for the orientable two-fold branched covers in Lemma~\ref{LEM:BR_COV} (see also \cite{HarveySearle}). For the sake of completeness, we include a proof.


 \begin{proposition}
 \label{PRO:BR_COV_GEOM}
 Let $X$ be a closed, three-dimensional Alexandrov space with curvature bounded below by $k$,  and assume that $X$ is not a topological manifold. If $Y$ is the orientable two-fold branched cover of $X$ in Lemma~\ref{LEM:BR_COV}, then the following hold:
 \begin{enumerate}
 	\item The metric in $X$ can be lifted to $Y$, so that  $Y$ is an Alexandrov  space with curvature bounded below by $k$.
	\item The involution $\iota:Y\to Y$ is an isometry and $Y/\iota$ is isometric to $X$.
	\item Let $p'\in Y$ be a fixed point of the involution $\iota:Y\to Y$. Then the space of directions $\Sigma_{p'}Y\cong \s^2$ is the canonical Alexandrov double cover of $\Sigma_{\proj(p')}X\cong \RP^2$.
 \end{enumerate}
 \end{proposition} 

Before proving the proposition, let us recall some definitions from Li \cite{Li}. Let $U$ be a length space. An open domain $\Omega\subset U$,  possibly incomplete, is called a \emph{$k$-domain} if the triangle or four-point comparison inequalities hold for any triangle or quadruple in $\Omega$. 
A length space $U$ is said to have \emph{curvature locally bounded from below by $k$} if, for any $p\in U$, there is a $k$-domain $\Omega_p\subset U$ containing $p$. Finally, we say that $U$ is \emph{weakly almost-everywhere convex} if for any $p\in U$  and any $\varepsilon >0$, there exists $p_1\in B_\varepsilon(p)$ for which the set of points $q\in U$ such that any geodesic $[p_1q]$ connecting $p_1$ to $q$ is not entirely contained in $U$ has $n$-dimensional Hausdorff measure zero.


\begin{proof}[{\bf Proof of Proposition~\ref{PRO:BR_COV_GEOM}}] Let $X$ be a closed three-dimensional Alexandrov space with curvature bounded below by $k$ and suppose that $X$ is not homeomorphic to a manifold. Then the set  $X'$ of points in $X$ whose space of directions is homeomorphic to $\mathbb{R}P^2$ is non-empty. Let $X_0=X\setminus X'$ be the set of manifold points of $X$. 

First, observe that $X_0$ is a convex subset of $X$. To see this, fix points $p,q\in X_0$ and join them by a minimal geodesic $\gamma$ in $X$. Let $x$ be a point in the interior of $\gamma$. Then $\Sigma_x$ is homeomorphic to $\s^2$, since $\mathrm{diam}\Sigma_x=\pi$ or, alternatively, because  the spaces of directions of points in the interior of a geodesic are isometric (see \cite{Petrunin-GAFA}) while the number of topologically singular points is finite. 

Let $d_{X_0}$ be the metric on $X_0$ given by the restriction of the metric on $X$ to $X_0$, i.e. $d_{X_0}=\left. d_X\right|_{X_0}$. Since $X_0\subset X$ is convex, the metric space $(X_0,d_{X_0})$ is a non-complete length space that is also a $k$-domain. On the other hand, $X_0$ is a non-orientable topological $3$-manifold. Let $Y_0$ be the orientable double cover of $X_0$ and equip $Y_0$ with the lifted metric, denoted $d_{Y_0}$. By construction, the metric $d_{Y_0}$ is locally isometric to $d_{X_0}$ and has curvature locally bounded below by $k$. Moreover,  the metric completion of $Y_0$ is homeomorphic to the two-fold branched cover $Y$ of $X$. Let $Y'$ be the preimage of $X'$, i.e. $Y'$ is the set of points that project down to points in $X$ whose space of directions is homeomorphic to $\RP^2$.

The metric space $(Y_0,d_{Y_0})$ is a length space (see \cite[Example 3.4.3]{BBI}). To prove that the metric completion $(Y,d_Y)$ of $(Y_0,d_{Y_0})$ also has curvature bounded below by $k$ we use work of Li ~\cite{Li}; more precisely, we verify that $(Y_0,d_{Y_0})$ is almost everywhere convex.

 Let $p' \in Y'\subset Y$ be a fixed point of the involution $\iota: Y \to Y$. By construction, for each geodesic $\gamma$ emanating from $\proj(p') \in X' \subset X$, there are exactly two  geodesics in the double branched cover $Y$ emanating from the point $p' \in Y'$  that lift $\gamma$. 
By looking at geodesics emanating from $p'$, we obtain the following: a locally isometric branched cover  $T_{p'}Y\to T_{\proj(p')}X$ with branching set the origin, a space of directions of $Y$ at $p'$  (as the closure under the angle distance of the unit geodesics) and a two-fold locally isometric covering map
of $\Sigma_{p'}Y$ to $\Sigma_{\proj(p')}X$ between
 the corresponding spaces of directions at $p' \in Y'$ and $\proj(p') \in X'$. This implies that $\Sigma_{p'}Y$ is an Alexandrov space with the same curvature bound as $\Sigma_{\proj(p')}X$ and proves part (3) of the proposition.

Suppose that $\gamma:[-\varepsilon,\varepsilon]\to Y$ is a unit geodesic with $\gamma(0)=p' \in Y'$.
Then the directions $\gamma'(0^+)$ and $\gamma'(0^-)$ form an angle $\pi$ at $p'\in Y$. This implies that $\Sigma_{p'}Y$ is a suspension with poles the directions  $\gamma'(0^+)$ and $\gamma'(0^-)$. The involution $\iota:Y\to Y$ whose quotient is $X$ induces an isometric involution of $\Sigma_{p'}Y$, and consequently must send $\gamma'(0^+)$ to $\gamma'(0^-)$ and vice versa. This implies that $\proj\circ\gamma(t)=\proj\circ\gamma(-t)$, for small $t$. As a consequence, it follows that for any $p\in Y$ there is at most one geodesic $\gamma:I\to Y$ starting at $p$ that contains a point $p' \in Y'$ in its interior.

Given $p\in Y_0$, let $\mathcal{C}_p=\{q\in Y_0 \mid [pq]\subset Y_0\}$. Then $Y_0\setminus \mathcal{C}_p$ is the set of points in $Y_0$ that are connected to $p$ via a minimal geodesic with respect to $d_Y$ that contains some point $p'\in Y'$ in its interior. Therefore, $Y_0\setminus \mathcal{C}_p$ is composed of ``lifts'' (``reflection'') of a geodesic joining $\proj(p)\in X$ with different singular points $\proj(p')\in X'$.  Since the covering map is a local isometry almost everywhere, it follows that these geodesic arcs have measure zero, and hence $Y_0\setminus \mathcal{C}_p$ has measure zero. Hence, by \cite[Corollary 0.1]{Li}, the metric completion $(Y,d_Y)$ is an Alexandrov space with curvature bounded below by $k$. This proves part (1) of the proposition.

To see that the involution $\iota:Y\to Y$ is a global isometry observe first that, by construction, $\iota$ is a bijective map. Since the branched cover $\proj:Y\to X$ is a local isometry everywhere, the involution $\iota:Y\to Y$ is also a local isometry, and hence a global isometry. This proves part (2) of the proposition.
\end{proof}


\subsection{Lower Ricci curvature bounds for metric-measure spaces}
\label{SEC:WEAK_RICCI}
In this subsection we turn our attention to general metric measure spaces with weak lower Ricci curvature bounds. In the first part we recall several notions of such curvature bounds and in the second we consider their behavior under lifts and coverings. We follow the approach of Lott-Sturm-Villani via optimal transport  applied to the $n$-dimensional Hausdorff measure, which is the natural reference measure on an $n$-dimensional Alexandrov space (see \cite{MondinoCavalletti}). For a different definition of lower Ricci curvature bound for Alexandrov spaces, conjecturally equivalent to the $\CD(K,N)$ definition, see \cite{ZhuZhang}.


\subsubsection{Reduced Curvature Dimension Condition}
\label{SEC:CDSTAR}
Let $(X,d,m)$ denote a metric measure space consisting of a complete, separable metric space $(X, d)$ and a locally finite measure $m$ on $(X,\mathcal{B}(X))$, that is, the volume $m(B_r (x))$ of balls is finite for all $x \in X$ and all sufficiently small $r >0$.

We denote by $(P^2(X), W_2)$ the $L^2$-Wasserstein space of probability measures $\mu$ on $(X,\mathcal{B}(X))$ with finite second moments. The $L^2$-Wasserstein distance $W_2(\mu_0, \mu_1)$ between two probability measures $\mu_0, \mu_1 \in P^2(X)$ is defined as
\[
W_2(\mu_0, \mu_1) = \inf \left(\int_{X \times X} d^2(x,y) ~d\pi(x,y)\right)^{1/2}, 
\]
where the infimum is taken over all couplings $\pi$ of $\mu_0$ and $\mu_1$ which are probability measures on $X \times X$ with marginals $\mu_0$ and $\mu_1$. The subspace of $m$-absolutely continuous measures is denoted by $\cP^2(X, d,m)$ and the subspace of $m$-absolutely continuous measures with bounded support is denoted  by $\cP_{\infty}(X, d,m)$.

The  $\CD^*(K,N)$ condition for $(X,d,m)$ is stated in terms of convexity of a certain entropy functional evaluated along geodesics in the Wasserstein space \cite{LottVillani, SturmI, SturmII}. For general $K \in \RR$ and $N \in (1,\infty]$, we introduce the following volume distortion coefficients. For $\theta \in \R^+$ and $t \in [0,1]$, set
\[
\sigma^{(t)}_{K,N}(\theta) :=
\begin{cases}
\quad\quad\quad \infty & \text{ if } K \theta^2 \geq N \pi^2,\\[.1cm]
\dfrac{\sin \left(t \theta \sqrt{K/N}\right)}{\sin \left( \theta \sqrt{K/N}\right)} & \text{ if } 0 < K\theta^2 < N\pi^2,\\[.1cm]
\quad\quad\quad t & \text{ if } K = 0,\\[.1cm]
\dfrac{\sinh \left(t \theta \sqrt{-K/N}\right)}{\sinh \left( \theta \sqrt{-K/N}\right)} & \text{ if } K < 0.  
\end{cases}
\]


\begin{definition}[{\em Reduced $(K,N)$-Curvature Dimension condition}]
\label{DEF:CDSTAR}
A metric measure space $(X,d,m)$ verifies $\CD^*(K,N)$ if and only if for each pair of measures $\mu_0, \mu_1 \in \cP_{\infty}(X,d,m)$ there exits an optimal coupling $\pi$ of $\mu_0=\rho_0m$ and $\mu_1= \rho_1m$ and a geodesic $\mu_t = \rho_t m \in \cP_{\infty}(X,d,m)$ such that
\be
\label{CD}
-\int_{X} \rho_t^{1-1/N'} dm \leq -\int_{X \times X} \left[\sigma_{K,N'}^{(1-t)}d(x_0,x_1)\rho_0^{-1/N'}(x_0) + 
 \sigma_{K,N'}^{(t)}d(x_0,x_1)\rho_{1}^{-1/N'}(x_1)\right] d\pi(x_0,x_1)
\ee
for all $t \in [0,1]$ and $N' \geq N$. 
\end{definition}

The original $\CD(K,N)$ condition of Lott-Sturm-Villani was defined in a similar way, replacing the volume distortion coefficients $\sigma^{(t)}_{K,N}(\cdot)$ and $\sigma^{(1-t)}_{K,N}(\cdot)$ above by the slightly larger coefficients $\tau^{(t)}_{K,N}(\cdot)$ and $\tau^{(1-t)}_{K,N}(\cdot)$, respectively, where 
\[
\tau^{(t)}_{K,N} (\theta) := t^{1/N} \sigma^{(t)}_{K,N-1}(\theta)^{(N-1)/N}.
\]
For $K>0$,
\[
\CD(K,N) \implies \CD^*(K,N) \implies \CD(K^*, N),
\]
with  $K^*=K(N-1)/N$ (see \cite{BacherSturm}). Note that, when $K=0$, the condition $\CD^*(0,N)$ is equivalent to $\CD(0,N)$ since $\tau^{(t)}_{0,N} (\theta) = \sigma^{(t)}_{0,N}  = t$.

 The reduced curvature-dimension condition $\CD^*(K,N)$ was introduced by Bacher-Sturm \cite{BacherSturm} because it satisfies the local-to-global property and tensorization for non-branching spaces. 
A metric measure space $(X, d, m)$ is \emph{non-branching} if $(X, d)$ is a geodesic metric space such that, given any $4$-tuple $(z, x_0, x_1, x_2)$ of points in $X$, if $z$ is both the midpoint of $[x_0x_1]$ and of $[x_0x_2]$, then $x_1 = x_2$.


\begin{theorem} [\protect{see \cite[Theorem 5.1]{BacherSturm}}]
Let $K,N \in \mathbb{R}$ with $N \geq 1$ and let $(X, d,m)$ be a non-branching metric measure space. Assume, additionally, that $\cP_{\infty}(X, d,m)$ is a geodesic space. Then $(X, d,m)$ satisfies $\CD^*(K,N)$ globally if and only if it satisfies $\CD^*(K,N)$ locally.
\end{theorem}

All Alexandrov spaces $(X,d)$ are non-branching and $\cP_{\infty}(X, d,\Haus^N)$ is a geodesic space.  


\begin{remark}
It is not known whether the $\CD(K,N)$ condition possesses the local-to-global property, see \cite{DengSturm} in regards to tensorization of $\CD(K,N)$-spaces. 
\end{remark}

Another notable extension of the original curvature-dimension condition came from Ambrosio, Gigli and Savar\'{e} \cite{AGS}, who introduced a more restrictive condition which retains the stability under measured Gromov-Hausdorff convergence but rules out Finsler geometries.  This stronger condition requires that the space is also infinitesimally Hilbertian (equivalently, the Laplacian on $X$ is a linear operator). Such spaces are denoted $\RCD(K,N)$-spaces, short for Riemannian Curvature Dimension condition. Naturally, an infinitesimally Hilbertian $\CD^*(K,N)$ space is denoted a $\RCD^*(K,N)$ space. In \cite{GKO}, Gigli, Kinderlehrer and Ohta proved linearity of the Laplacian on Alexandrov spaces; thus, all $\CD^*(K,N)$-Alexandrov spaces are $\RCD^*(K,N)$ metric measure spaces. The reverse implication is unknown, though it is conjectured to be true in dimension two. 


\begin{conjecture}[K.T.~Sturm]
Any $\RCD^*(K,2)$ metric measure space is an Alexandrov space of curvature bounded below.
\end{conjecture}


\subsubsection{Lifts of metric measure spaces}
\label{SEC:MMS_LIFTS} 
We conclude this section by recalling properties of universal coverings of metric measure spaces following \cite[Section 7]{BacherSturm}. Consider a metric measure space $(X,d,m)$ which satisfies the reduced curvature-dimension $\CD^*(K,N)$ locally for $K\geq 1$ and $N \geq 1$. The existence of a universal cover is guaranteed provided $X$ is connected, locally path connected, and semi-locally simply connected {\color{blue} \cite[Ch.~V, Theorem~10.2]{Massey}}. Let $\mathsf{p}: \widetilde{X} \to X$ denote such a universal cover and its covering map. Then $\widetilde{X}$ inherits a length structure from $X$ as follows:


\begin{definition}
A curve $\widetilde{\gamma}$ in $\widetilde{X}$ is called \emph{admissible} if its composition with $\mathsf{p}$ is a continuous curve in $X$. 
\end{definition}

The length of an admissible curve in $\widetilde{X}$ is set to be $\text{Length}(\widetilde{\gamma}) = \text{Length}(\mathsf{p} \circ \gamma)$ with respect to $X$. For two points $x,y \in \widetilde{X}$, define the associated distance $\widetilde{d}$ on $\widetilde{X}$ by 
\[
\widetilde{d}(x,y) = \inf \{\,\text{Length}(\widetilde{\gamma}) ~|~ \widetilde{\gamma}: [0,1] \to \widetilde{X} \text{ is admissible}, \widetilde{\gamma}(0) = x, \widetilde{\gamma}(1) = y\,\}.
\]
Let $\xi$ denote the family of all sets $\widetilde{E} \subset \widetilde{X}$ such that the restriction of $\mathsf{p}$ to $\widetilde{E}$ is a local isometry from $\widetilde{E}$ to a measurable set $E \subset X$. We define a function $\widetilde{m}: \xi \to [0, \infty)$ by
\[
\widetilde{m}(\widetilde{E}) = m(p(\widetilde{E})) = m(E)
\]
and extend $\widetilde{m}$ in a unique way to all Borel measurable sets on $\widetilde{X}$. We call the metric measure space $(\widetilde{X}, \widetilde{d}, \widetilde{m})$ the \emph{lift} of $(X,d,m)$. Because the construction of $\widetilde{m}$ and $\widetilde{d}$ is local, the properties of $(X,d,m)$ are lifted directly to $(\widetilde{X}, d,m)$. We summarize this in the following result.


\begin{theorem}[cf.~{\cite[Theorem 7.10]{BacherSturm}}]
Assume that $(X,d,m)$ is a non-branching metric measure space satisfying $\CD^*(K,N)$ locally for $K \in \mathbb{R}$ and $N \geq 1$ and that $(X,d)$ is semi-locally simply connected. Let $\widetilde{X}$ be the universal covering of $X$. Then $(\widetilde{X}, \widetilde{d}, \widetilde{m})$ satisfies $\CD^*(K,N)$ locally as well. 
\end{theorem}

Note that Perelman's conical neighborhood theorem \cite{Perelman-Morse} implies that any Alexandrov space $X$ has a universal covering $\widetilde{X}$ and the Hausdorff measure of $X$  lifts to the Hausdorff measure of $\widetilde{X}$. Thus, the universal cover of any $\CD^*(K,N)$-Alexandrov space is again a $\CD^*(K,N)$-Alexandrov space. 


\section{Properties of $\CD^*(K,N)$-Alexandrov spaces.}
\label{SEC:Alex+Ricci}

In this section, we provide a few facts of $\CD^*(K,N)$-Alexandrov spaces needed later on. Recall first the following non-smooth shortening lemma of Villani (cf. \cite[Theorem 8.22]{Villani}).


\begin{lemma}
\label{LEM:VILLANI_TWO_GEODESICS}
Let $(X,d)$ be a metric space, and let $\gamma_1, \gamma_2$ be two (constant-speed) minimizing geodesics such that 
\[
d(\gamma_1(0), \gamma_1(1))^2+d(\gamma_2(0), \gamma_2(1))^2\leq d(\gamma_1(0), \gamma_2(1))^2+d(\gamma_2(0), \gamma_1(1))^2
\]
Let $L_1$ and $L_2$ stand for the respective lengths of $\gamma_1$ and $\gamma_2$, and let $D$ be a bound on the diameter of $(\gamma_1\cup\gamma_2)([0,1])$. Then, for any $t_0\in(0,1)$,
\[
|L_1-L_2|\leq \frac{C\sqrt{D}}{\sqrt{t_0(1-t_0)}}\sqrt{d(\gamma_1(t_0), \gamma_2(t_0))},
\]
for some numeric constant $C$.
\end{lemma}

Using this lemma, we prove the following special case of Open Problem 8.21 in \cite{Villani}.


\begin{lemma}
\label{LEM:NO_COLLAPSE}
Let $(X,d)$ be an Alexandrov space with curvature bounded below by $K\in\mathbb{R}$, and let $x_1, x_2, y_1, y_2$ be four points in $X$ such that
\begin{equation}
d(x_1, y_1)^2+d(x_2, y_2)^2\leq d(x_1, y_2)^2+d(x_2, y_1)^2.
\end{equation}
Let $\gamma_1$ and $\gamma_2$ be two (constant-speed) geodesics respectively joining $x_1$ to $y_1$ and $x_2$ to $y_2$. If $\gamma_1(t_0)=\gamma_2(t_0)$ for some $t_0\in (0,1)$, then $\gamma_1(t)\equiv\gamma_2(t)$ for all $t\in[0,1]$.
\end{lemma}


\begin{proof}
Lemma~\ref{LEM:VILLANI_TWO_GEODESICS} implies that the geodesics $\gamma_1$ and $\gamma_2$ have the same length $L$.
By the triangle inequality, we have
\begin{align*}
d(x_1, y_2) & \leq d(x_1, \gamma_1(t_0))+d(\gamma_1(t_0), y_2)=t_0L+(1-t_0)L=L,\\
d(x_2, y_1) & \leq d(x_2, \gamma_2(t_0))+d(\gamma_2(t_0), y_1)=t_0L+(1-t_0)L=L.
\end{align*}
These inequalities imply that
\begin{align*}
2L^2  & =d(x_1, y_1)^2+d(x_2, y_2)^2\\
		& \leq d(x_1, y_2)^2+d(x_2, y_1)^2\\
		& \leq 2L^2.
\end{align*}
Hence $d(x_1, y_2)=d(x_2, y_1)=L$.
It follows that $\gamma_1([0, t_0])\cup \gamma_2([t_0, 1])$ is a minimal geodesic from $x_1$ to $y_2$. Since geodesics in an Alexandrov space do not branch, we must have that $\gamma_1=\gamma_2$ on $[t_0,1]$. To see that  $\gamma_1=\gamma_2$ on $[0,t_0]$, consider  the geodesics 
\begin{align*}
	\widetilde{\gamma}_1 & = \gamma_1(1-t),\quad t\in [0,1], \\
	\widetilde{\gamma}_2 & = 	\begin{cases}
										& \gamma_1(1-t), \quad t\in [0,1-t_0]\\
										& \gamma_2(1-t),\quad t\in [1-t_0,1]
									  	\end{cases}.
\end{align*}
Once again, non-branching implies that $\gamma_1(1-t)=\gamma_2(1-t)$ for $t\in [1-t_0,1]$. That is, $\gamma_1=\gamma_2$ on $[0,t_0]$, as we wanted to show. 
\end{proof}

We use Lemma~\ref{LEM:NO_COLLAPSE} to prove the following result, which will play a crucial role in the proof of Theorem~\ref{THM:RICCI_POS}.


\begin{proposition}
\label{PRO:CD_FILL}
Let $X$ be an Alexandrov space and fix $x_o \in X$. For $K \in \RR$ and $N > 1$, if $X \setminus \{x_o\}$ is a $\CD^*(K,N)$-Alexandrov space, then so is $X$. 
\end{proposition}


\begin{proof}
Let $X_o = X \setminus \{x_o\}$ and note that any measure $\mu \in \cP_{\infty}(X, d, \Haus^N)$ can naturally be defined as a measure $\mu^o$ on $X_o$ by letting $\mu^o(E) := \mu(E\cap X_o)$. Observe that $\mu^o \in \cP_{\infty}(X_o, \left. d\right|_{X_o}, \left.\Haus^N\right|_{X_o})$, since $\Haus^N(\{x_o\}) = 0$. 
Let $\mu_0,\mu_1$ be measures in $\cP_{\infty}(X, d, \Haus^N)$.
Since $\Haus^N(\{x_o\}) = 0$, the measures $\mu_0^o$ and $\mu_1^o$ are absolutely continuous with respect to the Hausdorff measure on $X_o$ as well.  Let $\mu_t$ be a geodesic in $ \cP_{\infty}(X, d, \Haus^N)$ from $\mu_0$ to $\mu_1$ and let $\mu^o_t$ be the corresponding curve in $\cP_{\infty}(X_o, \left. d\right|_{X_o}, \left.\Haus^N\right|_{X_o})$ between the measures $\mu_0^o$ and $\mu_1^o$. We would like to show that $\mu^o_t$ is a geodesic between $\mu_0^o$ and $\mu_1^o$ in  $\cP_{\infty}(X_o, \left. d\right|_{X_o}, \left.\Haus^N\right|_{X_o})$.

Since $\mu_t$ is optimal, we can write $\mu_t = (e_t)_{\#}\Pi$ where $\Pi$ is a optimal dynamic transference plan between $\mu_0, \mu_1 \in \cP_{\infty}(X)$ and $e_t: \Gamma \to X$ is the usual evaluation map from the set of minimizing geodesics on $X$.  As the mass of $\mu_0$ is transported along geodesics to $\mu_1$, it follows from Lemma \ref{LEM:NO_COLLAPSE} that if any geodesics intersect, then they must coincide for all times. Therefore, $\mu_t(\{x_o\}) = 0$, for all $t \in [0,1]$, whenever $x_o \in \supp(\mu_t)$. Thus, the density of $\mu_t$ with respect to $\Haus^N$ is equal to the density of $\mu_t^o$ with respect to $\left.\Haus^N\right|_{X_o}$. That is, if $\mu_t  = \rho_t \Haus^N$ and $\mu_t^o= \rho^o_t\left.\Haus^N\right|_{X_{o}}$, then $\rho_t^o=\left.\rho_t \right|_{X_{o}}$.

Observe now that the inclusion $i:X_o\hookrightarrow X$ induces a map
\[
i_{\#}:\cP_{\infty}(X_o, W_2^{X_o})\hookrightarrow
\cP_{\infty}(X, W_2^X)
\]
via the pushforward. On the other hand, we have a map
\[
j:\cP_{\infty}(X, W_2^X)\hookrightarrow
\cP_{\infty}(X_o, W_2^{X_o})
\]
sending $\mu$ to $\mu^o$. Observe  that $j$ and $i_\#$ are inverses to each other. 

In order to prove that the  inclusion $i_\#$ is an isometry, recall that Proposition~3.8 in \cite{Bertrand} implies, since $X$ is an Alexandrov space, that there is a measurable function $F:X\to X$ with $F_{\#}\mu_0=\mu_1$ such that 
\[
W_2^X(\mu_0,\mu_1)= \left(\int_{X\times X}d^2(x,y)\,d\Pi(x,y)\right)^{1/2},
\]
with
\[
\Pi=(\mathrm{Id}, F)_{\#}\mu_0. 
\]
Observe that
\begin{equation}
\label{EQ:PUNCTURED_1}
	(\mathrm{Id},F)_\#\mu_0(\{x_o\}\times X)  = \mu_0(\{x_o\}) =0
\end{equation}
and
\begin{align}
\label{EQ:PUNCTURED_2}
	(\mathrm{Id},F)_\#\mu_0(X\times\{x_o\})  & = \mu_0(F^{-1}(\{x_o\})) \nonumber \\ 
																& = F_\#\mu_0(\{x_o\})\\ \nonumber
																& = \mu_1(\{x_o\})= 0. \nonumber
\end{align}

To adapt $F$ to $X_o$ we define a function $G: X_o\to X_o$ given by
\[
G(x)=
\begin{cases}
	F(x) & \text{if } F(x)\neq x_o\\
	x_1  & \text{if } F(x) = x_o,
\end{cases}
\]
where $x_1\neq x_o$ is an arbitrary point in $X$. 
Observe that 
\[
G_\#(\mu_0^o)=\mu_1^o.
\] 
Now let 
\[
\Pi^o=(\mathrm{Id},G)_\#(\mu_0^o)
\]
so that
\begin{equation}
\label{EQ:PUNCTURED_3}
(i\times i)_\#\Pi^o = \Pi.
\end{equation}
We have 
\begin{align}
\label{EQ:PUNCTURED_4}
	W_2^{X}(\mu_0,\mu_1) & = \left(\int_{X\times X}d^2(x,y)\,d\Pi(x,y)\right)^{1/2}\nonumber\\[.3cm]
	& =\left( \int_{X_o\times X_o}d^2(x,y)\,d\Pi^o(x,y)\right)^{1/2} \\[.3cm]
	& \geq W_2^{X_o}(\mu^o_0,\mu^o_1),\nonumber
\end{align}
where we have used Equations \eqref{EQ:PUNCTURED_1}, \eqref{EQ:PUNCTURED_2} and \eqref{EQ:PUNCTURED_3}. 

Now, let $\Lambda^o$ be a measure on $X_o\times X_o$ that realizes the distance $W^{X_o}_2(\mu^o_0,\mu_1^o)$.  Let $\Lambda$ be the measure on $X\times X$ given by 
\[
(i\times i)_\#\Lambda^o = \Lambda.
\]
A simple computation shows that the marginals of $\Lambda$ are $\mu_0$ and $\mu_1$. We have
\begin{align}
\label{EQ:PUNCTURED_5}
	W_2^{X_o}(\mu^o_0,\mu^o_1) & = \left(\int_{X_o\times X_o}d^2(x,y)\,d\Lambda^o(x,y)\right)^{1/2}\nonumber\\[.3cm]
	& =\left( \int_{X\times X}d^2(x,y)\,d\Lambda(x,y)\right)^{1/2} \\[.3cm]
	& \geq W_2^{X}(\mu_0,\mu_1),\nonumber
\end{align}
by the change of variables formula for the pushforward. It follows from Equations \eqref{EQ:PUNCTURED_4} and \eqref{EQ:PUNCTURED_5} that 
\[
W_2^{X_o}(\mu^o_0,\mu^o_1)  =   W_2^{X}(\mu_0,\mu_1).
\]
It follows that a geodesic in $\cP_{\infty}(X, W_2^{X})$ is also a geodesic in $\cP_{\infty}(X_o, W_2^{X_o})$ and viceversa. Therefore, since the geodesic $\mu^o_t=\rho^o_t\mathcal{H}^N$ satisfies condition \eqref{CD} and $\rho^o_t = \left. \rho_t\right|_{X_o}$, we have that $\mu_t=\rho_t\mathcal{H}^N$ satisfies condition \eqref{CD}. Therefore $X$ is a $\CD^*(K,N)$ space.
\end{proof}


\begin{remark}
In fact, Proposition~\ref{PRO:CD_FILL} should still hold for $X \setminus X'$ where $X'$ is any subset of zero Hausdorff measure.
\end{remark}


\section{$\CD^*(2,3)$-Alexandrov spaces}
\label{SEC:RICCI_POS}

We are now ready to show that a closed three-dimensional Alexandrov space satisfying the $\CD^*(2,3)$ condition is homeomorphic to a spherical space form or to the suspension of $\RP^2$.


\begin{proof}[\bf Proof of Theorem \ref{THM:RICCI_POS}]
Let $(X,d_X)$ be a closed three-dimensional Alexandrov space with curvature bounded below by $k \in \R$, and suppose its Hausdorff measure satisfies $\CD^*(K,3)$, for $K>0$. Note that by rescaling, we may assume $(X,d_X,\Haus^3)$ is, in fact, $\CD^*(2,3)$. Suppose first that $X$ is a topological manifold. By \cite[Theorem 7.10]{BacherSturm}, $X$ has finite fundamental group. Therefore, by Perelman's resolution of Thurston's Elliptization Conjecture, $X$ must be homeomorphic to a spherical space form. 

Suppose now that $X$ is not a topological manifold. Let $X'$ denote the points in $X$ which are topologically singular. By assumption $X'$ is nonempty and, by compactness, the set $X'$ is finite. Denote by $x_1, x_2, \dots, x_k$ the points in $X'$. After removing a neighborhood of each $x_j$, we obtain a non-orientable topological 3-manifold with $k$ boundary copies of $\RP^2$. 

Let $\proj: Y \to X$ be the orientable, double branched cover of $X$ with branching set $X'$ as in Lemma \ref{LEM:BR_COV}. By Proposition \ref{PRO:BR_COV_GEOM}, we can lift the metric $d_X$ to a metric $d_Y$ so that $Y$ is an Alexandrov space with curvature bounded below by $k$ as well. Let $y_i = \proj^{-1}(x_i)$ for $i = 1,2,\dots, k$ be the isolated fixed points of the involution $\iota: Y \to Y$ and denote this set by $Y'$. As in the proof of Proposition \ref{PRO:BR_COV_GEOM}, consider $Y_0= Y\setminus Y'$ with the restricted metric $d_{Y_0}:= d_Y |_{Y_0}$ and equipped with the lifted Hausdorff measure from $X_0$ as described in Section \ref{SEC:MMS_LIFTS}. Note that the lift of the Hausdorff measure on $X_0$ is precisely the Hausdorff measure on $Y_0$ and $(Y_0, d_{Y_0}, \Haus^3)$ is locally $\CD^*(2,3)$ since we know $(X,d_X,\Haus^3)$ is. Hence, $Y_0$ is $\CD^*(2,3)$ globally. Therefore, it follows from Proposition \ref{PRO:CD_FILL} that $Y$ is $\CD^*(2,3)$ and thus, by \cite[Theorem 7.10]{BacherSturm}, the  fundamental group of $Y$ is finite. The remainder of the proof follows as in \cite{GGG}. We include below the main ideas of the rest of the argument for completeness.

Briefly, since $\pi(Y\setminus Y') \simeq \pi(Y)$, the subgroup $\proj_*(\pi(Y\setminus Y'))$ is an index 2 subgroup in $\pi(X\setminus X')$. Thus, $\pi(X\setminus X')$ is finite.  It then follows from Epstein's theorem {\cite[Theorem 9.6]{Hempel}} and  Perelman's proof of Poincar\'e conjecture that $X\setminus X'$ is homeomorphic to $\RP^2 \times [0,1]$ and thus $k=2$. This proves the Theorem. 

\end{proof}


\section{$\CD^*(0,3)$-Alexandrov spaces}
\label{SEC:RICCI_NONNEG}
We classify now closed three-dimensional Alexandrov spaces with nonnegative Ricci curvature by following closely the argument in \cite{GGG}. In that reference the Splitting Theorem for nonnegatively curved Alexandrov spaces was a key tool. Since we work with nonnegative Ricci curvature, we must instead rely on Gigli's Splitting Theorem for $\RCD(0,N)$ metric measure spaces \cite{Gigli-Splitting} as well as Proposition \ref{PRO:CD_FILL}.


\begin{theorem}[Gigli]
\label{THM:RICCI-NONNEG}
Let $(M, d, m)$ be an $\RCD(0, N)$ space containing a line. Then $(M, d, m)$ is isomorphic to the product of the Euclidean line $(\RR, d_{\RR}, \mathcal{L}^1)$ and another space $(M', d', m')$, where the product distance $d' \times d_{\RR}$ is defined as
\[
d' \times d_{\RR}\left( (x,t),(y,s)\right) := d'(x',y') + |t-s|^2, \qquad \forall x', y' \in M' \text{ and } t,s, \in \RR.
\]
\noindent Moreover, if $N\geq 2$,  $(M',d',m')$ is a $\RCD(0,N-1)$ space. Here \emph{isomorphic} means that there is a measure preserving isometry between the spaces.
\end{theorem}

We are now ready to prove the theorem.


\begin{proof}[\bf Proof of Theorem \ref{THM:RICCI_NONNEG}]
Recall that the coefficients for the reduced curvature dimension condition $\CD^*(0,N)$ coincide with those of the usual $\CD(0,N)$, so we can work with any of the two conditions. We consider two possibilities, depending on whether or not $X$ is a topological manifold. \\


\noindent \textbf{Case 1:} \emph{$X$ is a topological manifold.} In this case, there are two possibilities: either the fundamental group $\pi_1(X)$ is finite or not. In the former case, $X$ is homeomorphic to a spherical space form, by Perelman's resolution of Thurston's Elliptization Conjecture, so we focus our attention on the case where $\pi_1(X)$ is infinite.  

Let $\widetilde{X}$ denote the universal cover of $X$ equipped with the lifted metric and measure. Namely, $\widetilde{X}$ is a $\CD(0,3)$-Alexandrov space and contains a line since $|\pi_1(X)| = \infty$. By Gigli's splitting theorem, $\widetilde{X}$ splits isometrically as a product $\mathbb{R} \times \widetilde{X}'$, where $\widetilde{X}'$ is also $\RCD(0, 2)$. It follows that $\widetilde{X}'$ is a simply-connected $2$-dimensional Alexandrov space. Thus, $\widetilde{X}'$ is either $\mathbb{S}^2$ or $\mathbb{R}^2$.

Suppose first that $\widetilde{X}'$ is a topological $2$-sphere. It follows from \cite{To} that $X$ will be homeomorphic to either $\sphere^2\times\sphere^1$ or $\RP^3\#\RP^3$, if $X$ is orientable, or to $\sphere^1\times\RP^2$ or $\sphere^2\tilde{\times}\sphere^1$, the nonorientable $\sphere^2$-bundle over $\sphere^1$, if $X$ is nonorientable.

Suppose now that $\widetilde{X}'$ is homeomorphic to $\RR^2$. Then, since $X$ is closed, the metric on $\widetilde{X}'$ has a compact quotient by isometries. We conclude, by applying again Gigli's Splitting Theorem, that $\tilde{X}'$ must be isometric to Euclidean two-dimensional space $\mathbb{E}^2$. It follows that  $\widetilde{X}$ is isometric to $\mathbb{E}^3$ and $X$ is isometric to a closed flat $3$-manifold.
\\


\noindent \textbf{Case 2:} \emph{$X$ is not a topological manifold.} Then there  exists a double branched cover $\iota: Y\to X$, where $Y$ is a closed orientable topological $3$-manifold with an Alexandrov metric of nonnegative Ricci curvature and with an isometric orientation reversing involution $\iota$ whose fixed points are the branching points of the double branched cover. As in \cite{GGG}, we obtain all  possible spaces $X$ by considering orientation reversing isometric involutions with isolated fixed points on closed, orientable Alexandrov $3$-manifolds $Y$ with nonnegative Ricci curvature. Therefore, $Y$ is one of the spaces appearing in Case 1. The arguments to conclude the proof are now the same as in \cite{GGG}.   
\end{proof}


\bibliography{Classification}

\def\cprime{$'$}
\providecommand{\bysame}{\leavevmode\hbox to3em{\hrulefill}\thinspace}
\providecommand{\MR}{\relax\ifhmode\unskip\space\fi MR }
\providecommand{\MRhref}[2]{%
  \href{http://www.ams.org/mathscinet-getitem?mr=#1}{#2}
}
\providecommand{\href}[2]{#2}
\begin{thebibliography}{10}

\bibitem{AGS}
Luigi Ambrosio, Nicola Gigli, and Giuseppe Savar{\'e}, \emph{Metric measure
  spaces with {R}iemannian {R}icci curvature bounded from below}, Duke Math. J.
  \textbf{163} (2014), no.~7, 1405--1490. \MR{3205729}

\bibitem{BacherSturm}
Kathrin Bacher and Karl-Theodor Sturm, \emph{Localization and tensorization
  properties of the curvature-dimension condition for metric measure spaces},
  J. Funct. Anal. \textbf{259} (2010), no.~1, 28--56.

\bibitem{BakryEmery}
D.~Bakry and Michel {\'E}mery, \emph{Diffusions hypercontractives}, S\'eminaire
  de probabilit\'es, {XIX}, 1983/84, Lecture Notes in Math., vol. 1123,
  Springer, Berlin, 1985, pp.~177--206. \MR{889476 (88j:60131)}

\bibitem{Bertrand}
J{\'e}r{\^o}me Bertrand, \emph{Existence and uniqueness of optimal maps on
  {A}lexandrov spaces}, Adv. Math. \textbf{219} (2008), no.~3, 838--851.
  \MR{2442054 (2009j:49094)}

\bibitem{BBI}
Dmitri Burago, Yuri Burago, and Sergei Ivanov, \emph{A course in metric
  geometry}, Graduate Studies in Mathematics, vol.~33, American Mathematical
  Society, Providence, RI, 2001. \MR{1835418 (2002e:53053)}

\bibitem{BGP}
Yu. Burago, M.~Gromov, and G.~Perelman, \emph{A. {D}. {A}leksandrov spaces with
  curvatures bounded below}, Uspekhi Mat. Nauk \textbf{47} (1992), no.~2(284),
  3--51, 222. \MR{1185284 (93m:53035)}

\bibitem{MondinoCavalletti}
Fabio Cavalletti and Andrea Mondino, \emph{Measure rigidity of {R}icci
  curvature lower bounds}, Adv. Math. \textbf{286} (2016), 430--480.
  \MR{3415690}

\bibitem{CheegerColding}
Jeff Cheeger and Tobias~H. Colding, \emph{On the structure of spaces with
  {R}icci curvature bounded below. {I}}, J. Differential Geom. \textbf{46}
  (1997), no.~3, 406--480. \MR{1484888 (98k:53044)}

\bibitem{ColdingNaber}
Tobias~Holck Colding and Aaron Naber, \emph{Characterization of tangent cones
  of noncollapsed limits with lower {R}icci bounds and applications}, Geom.
  Funct. Anal. \textbf{23} (2013), no.~1, 134--148. \MR{3037899}

\bibitem{DengSturm}
Qintao Deng and Karl-Theodor Sturm, \emph{Localization and tensorization
  properties of the curvature-dimension condition for metric measure spaces,
  {II}}, J. Funct. Anal. \textbf{260} (2011), no.~12, 3718--3725. \MR{2781974
  (2012i:53030)}

\bibitem{GGG}
Fernando Galaz-Garcia and Luis Guijarro, \emph{On three-dimensional
  {A}lexandrov spaces}, Int. Math. Res. Not. IMRN (2015), no.~14, 5560--5576.
  \MR{3384449}

\bibitem{Gigli-Splitting}
Nicola Gigli, \emph{An overview of the proof of the splitting theorem in spaces
  with non-negative {R}icci curvature}, Anal. Geom. Metr. Spaces \textbf{2}
  (2014), 169--213. \MR{3210895}

\bibitem{GKO}
Nicola Gigli, Kazumasa Kuwada, and Shin-Ichi Ohta, \emph{Heat flow on
  {A}lexandrov spaces}, Comm. Pure Appl. Math. \textbf{66} (2013), no.~3,
  307--331. \MR{3008226}

\bibitem{GW}
Karsten Grove and Burkhard Wilking, \emph{A knot characterization and
  1-connected nonnegatively curved 4-manifolds with circle symmetry}, Geom.
  Topol. \textbf{18} (2014), no.~5, 3091--3110. \MR{3285230}

\bibitem{Hamilton}
Richard~S. Hamilton, \emph{Three-manifolds with positive {R}icci curvature}, J.
  Differential Geom. \textbf{17} (1982), no.~2, 255--306. \MR{664497
  (84a:53050)}

\bibitem{HarveySearle}
John. Harvey and Catherine Searle, \emph{Orientation and symmetries of
  alexandrov spaces}, arXiv:1209.1366 [math.DG] (2013).

\bibitem{Hempel}
John Hempel, \emph{{$3$}-{M}anifolds}, Princeton University Press, Princeton,
  N. J.; University of Tokyo Press, Tokyo, 1976, Ann. of Math. Studies, No. 86.
  \MR{0415619 (54 \#3702)}

\bibitem{KS}
Kazuhiro Kuwae and Takashi Shioya, \emph{Infinitesimal {B}ishop-{G}romov
  condition for {A}lexandrov spaces}, Probabilistic approach to geometry, Adv.
  Stud. Pure Math., vol.~57, Math. Soc. Japan, Tokyo, 2010, pp.~293--302.
  \MR{2648266 (2011c:53083)}

\bibitem{Li}
Nan Li, \emph{Globalization with probabilistic convexity}, J. Topol. Anal.
  \textbf{7} (2015), no.~4, 719--735. \MR{3400128}

\bibitem{LV}
John Lott and C{\'e}dric Villani, \emph{Ricci curvature for metric-measure
  spaces via optimal transport}, Ann. of Math. (2) \textbf{169} (2009), no.~3,
  903--991. \MR{2480619 (2010i:53068)}

\bibitem{LottVillani}
\bysame, \emph{Ricci curvature for metric-measure spaces via optimal
  transport}, Ann. of Math. (2) \textbf{169} (2009), no.~3, 903--991.
  \MR{2480619 (2010i:53068)}

\bibitem{Massey}
William~S. Massey, \emph{A basic course in algebraic topology}, Graduate Texts
  in Mathematics, vol. 127, Springer-Verlag, New York, 1991. \MR{1095046
  (92c:55001)}

\bibitem{MondinoNaber}
Aaron Naber and Andrea Mondino, \emph{Structure theory of metric-measure spaces
  with lower ricci curvature bounds i}, arXiv:1405.2222v2[math.DG] (2014).

\bibitem{Perelman-Morse}
G.~Ya. Perel{\cprime}man, \emph{Elements of {M}orse theory on {A}leksandrov
  spaces}, Algebra i Analiz \textbf{5} (1993), no.~1, 232--241. \MR{1220498
  (94h:53054)}

\bibitem{Petrunin-GAFA}
Anton Petrunin, \emph{Parallel transportation for {A}lexandrov space with
  curvature bounded below}, Geom. Funct. Anal. \textbf{8} (1998), no.~1,
  123--148. \MR{1601854 (98j:53048)}

\bibitem{Petrunin}
\bysame, \emph{Alexandrov meets {L}ott-{V}illani-{S}turm}, M\"unster J. Math.
  \textbf{4} (2011), 53--64. \MR{2869253 (2012m:53087)}

\bibitem{RajalaSturm}
Tapio Rajala and Karl-Theodor Sturm, \emph{Non-branching geodesics and optimal
  maps in strong {$CD(K,\infty)$}-spaces}, Calc. Var. Partial Differential
  Equations \textbf{50} (2014), no.~3-4, 831--846. \MR{3216835}

\bibitem{SturmI}
Karl-Theodor Sturm, \emph{On the geometry of metric measure spaces. {I}}, Acta
  Math. \textbf{196} (2006), no.~1, 65--131. \MR{2237206 (2007k:53051a)}

\bibitem{SturmII}
\bysame, \emph{On the geometry of metric measure spaces. {II}}, Acta Math.
  \textbf{196} (2006), no.~1, 133--177. \MR{2237207 (2007k:53051b)}

\bibitem{To}
Jeffrey~L. Tollefson, \emph{The compact {$3$}-manifolds covered by
  {$S\sp{2}\times R\sp{1}$}}, Proc. Amer. Math. Soc. \textbf{45} (1974),
  461--462. \MR{0346792 (49 \#11516)}

\bibitem{Villani}
C{\'e}dric Villani, \emph{Optimal transport}, Grundlehren der Mathematischen
  Wissenschaften [Fundamental Principles of Mathematical Sciences], vol. 338,
  Springer-Verlag, Berlin, 2009, Old and new. \MR{2459454 (2010f:49001)}

\bibitem{ZhuZhang}
Huichun Zhang and Xiping Zhu, \emph{On a new definition of {R}icci curvature on
  {A}lexandrov spaces}, Acta Math. Sci. Ser. B Engl. Ed. \textbf{30} (2010),
  no.~6, 1949--1974. \MR{2778704 (2012e:53068)}

\end{thebibliography}
\bibliographystyle{amsplain}
\end{document}